\documentclass[11pt]{article}

\usepackage[margin=1in]{geometry}
\usepackage{amsmath,amssymb,amsthm}
\usepackage{enumitem}
\usepackage{xcolor}
\usepackage{hyperref}

\newcommand{\N}{\mathbb{N}}

\newcommand{\supp}{\mathrm{supp}}

\newcommand{\Hom}{\operatorname{Hom}}
\newcommand{\ord}{\operatorname{ord}}

\newtheorem{theorem}{Theorem}[section]
\newtheorem{lemma}[theorem]{Lemma}
\newtheorem{corollary}[theorem]{Corollary}
\newtheorem{definition}[theorem]{Definition}
\newtheorem{remark}[theorem]{Remark}
\newtheorem{proposition}[theorem]{Proposition}
\newtheorem{conjecture}[theorem]{Conjecture}

\title{The Separating Noether Number of Finite Abelian Groups}

\author{Jing Huang\footnote{E-mail address: {\it
jhuangmath@gzhu.edu.cn (J. Huang).
}}
}

\date{\small
School of Mathematics and Information Science,
Guangzhou University, Guangzhou 510006, China
}

\begin{document}
\maketitle

\begin{abstract}
For a finite abelian group $G$,
let $\beta_{\mathrm{sep}}(G)$ denote its
separating Noether number.
We determine $\beta_{\mathrm{sep}}(G)$
exactly for every finite abelian group
$
G \cong C_{n_1}\oplus \cdots \oplus C_{n_r}$
with
$
1<n_1 \mid \cdots \mid n_r.
$
If $r=2s-1$, then
\[
\beta_{\mathrm{sep}}(G)=n_s+n_{s+1}+\cdots+n_r,
\]
whereas if $r=2s$, then
\[
\beta_{\mathrm{sep}}(G)=\frac{n_s}{p_1}+n_{s+1}+\cdots+n_r,
\]
where $p_1$ denotes the smallest
prime divisor of $n_1$.
Our proof is additive-combinatorial in
nature.
It avoids the Davenport-equality assumption
$\mathsf{D}(n_sG)=\mathsf{D}^{*}(n_sG)$
used in previous works. The key ingredients are
a geometric reduction of auxiliary sequences via
the novel construction of geodesic surrogates,
alongside a uniform lifting procedure for relation groups.
As an application, we prove that if $r\ge 2$, then every extremal
separating atom $A$ over $G_0$ with $|G_0|\le r+1$
satisfies $|\supp(A)|=|G_0|=r+1$. Equivalently,
the conjectured support conclusion of Schefler, Zhao, and Zhong
holds for all finite abelian groups of rank at least $2$.
By contrast, the rank-$1$ case is exceptional:
for cyclic groups, the analogous conjectural conclusion is false.


\end{abstract}

\section{Introduction}
Let $G$ be a finite group acting linearly on a
finite-dimensional vector space $V$ over the
field of complex numbers $\mathbb{C}$. This
action extends naturally to the coordinate
ring $\mathbb{C}[V]$, which can be identified with
the polynomial algebra $\mathbb{C}[x_1,\ldots,x_k]$,
where $k=\dim_{\mathbb{C}}V$.
Noether's degree bound \cite{Noether16}
guarantees that the invariant ring
$\mathbb{C}[V]^G$ is finitely generated
by homogeneous polynomials of degree less
than or equal to $|G|$. The \emph{Noether number},
denoted by $\beta(G)$, is defined as the smallest
positive integer $d$ such that for every finite-dimensional
$G$-module $V$, the ring $\mathbb{C}[V]^G$ is
generated by homogeneous invariants of degree at most $d$
(see \cite{Schmid91}).

In the non-abelian setting, the exact value of the
Noether number is known for only a very
restricted number of group families (see  \cite{Cz19,Cz14,Cz18}).
For finite abelian groups, a decisive bridge to additive
combinatorics was built
by Schmid  \cite{Schmid91}, who showed that the Noether number coincides
with the Davenport constant:
$
\beta(G)=\mathsf{D}(G)
$
for every finite abelian group $G$.
In this way, the invariant-theoretic study of
$\beta(G)$ became closely tied to the arithmetic of
zero-sum sequences, a circle of ideas that goes back
to Olson's foundational work on zero-sum problems
\cite{OlsonI,OlsonII} and has been developed
systematically in the modern zero-sum literature (see, for instance, \cite{GaoGeroldinger06,GeroldingerHK06}).
Nevertheless, despite this powerful combinatorial
translation, computing the exact value of
$\beta(G)$ (equivalently $\mathsf{D}(G)$) remains notoriously difficult.
To date, precise formulas for $\mathsf{D}(G)$
have been established only for restricted classes,
primarily groups of rank at most $2$, finite
abelian $p$-groups, and a handful of
specialized families;
 we refer the reader to \cite[Lemma 2.1]{SZZ}
for a list of families for which the equality
$\mathsf{D}(G)=\mathsf{D}^{*}(G)$ is known
(and hence $\mathsf{D}(G)$ is explicit). For general
finite abelian groups of rank $3$ or
higher, the exact determination of the
Noether number remains a major open problem.

However, requiring a set of invariants to generate
the entire algebra $\mathbb{C}[V]^G$ is a
notoriously demanding condition, often leading
to generating sets of prohibitive size and complexity.
To circumvent these obstacles, a weaker but
often more flexible invariant-theoretic notion
is that of a \emph{separating set}.
A subset $S\subseteq \mathbb{C}[V]^G$ is separating
if for all $v,w\in V$ with $Gv\neq Gw$ there exists
$f\in S$ such that $f(v)\neq f(w)$.
This notion was promoted in computational invariant theory by Derksen and Kemper \cite{DerksenKemper02}.
In analogy with the Noether number,
Kohls and Kraft \cite{KohlsKraft10}
introduced the \emph{separating Noether number}
$\beta_{\mathrm{sep}}(G)$ of a finite
group $G$, defined as the smallest positive
integer $d$ such that for every
finite-dimensional $G$-module $V$,
there is a separating set consisting of
invariant polynomials of degree at most $d$.
By definition,
$
\beta_{\mathrm{sep}}(G)\le \beta(G),
$
and determining $\beta_{\mathrm{sep}}(G)$
is the orbit-separation analogue of the
classical Noether number problem.

For finite abelian groups, Domokos \cite{Domokos17}
reduced the computation of $\beta_{\mathrm{sep}}(G)$ to a finite additive-combinatorial problem on relation
groups and small supports. More precisely,
the computation is controlled by subsets
$G_0 \subseteq G$ of cardinality at
most $\mathsf{rank}(G)+1$
(the Helly dimension of $G$),
allowing the problem to be reformulated
in terms of separating atoms in
monoids of zero-sum sequences.

During the last few years, exact formulas
for $\beta_{\mathrm{sep}}(G)$ have been obtained
for several important families of finite abelian
groups. Domokos settled the cyclic case \cite{Domokos17}.
Schefler treated all rank $2$ groups \cite{SchJCTA}
and direct sums of copies of a cyclic group \cite{Sch25}.
Recently, Schefler, Zhao, and Zhong \cite{SZZ,SZZ2509}
developed a systematic approach via separating atoms,
yielding exact formulas for broad classes of groups,
including finite abelian $p$-groups and groups of rank $3$,
$4$,  and $5$. More generally, they obtained exact formulas
or sharp upper bounds on $\beta_{\mathrm{sep}}(G)$
under the additional hypothesis
$\mathsf{D}(n_sG)=\mathsf{D}^{*}(n_sG)$,
where $\mathsf{D}^{*}$ denotes the standard
structural lower bound for the Davenport
constant and $n_s$ is a suitably chosen invariant of $G$.
The equality
$
\mathsf{D}(n_sG)=\mathsf{D}^{*}(n_sG)
$
holds in several important
situations, such as  groups of rank at
most two and $p$-groups,
but not in full generality.
Unconditionally removing this hypothesis
is therefore the central challenge to obtaining
a complete formula for $\beta_{\mathrm{sep}}(G)$.

Alongside computing the exact global formula,
a fundamental objective is the corresponding inverse
problem: characterizing the structure of extremal
separating atoms of maximal length
$\beta_{\mathrm{sep}}(G)$. Since
Domokos' reduction restricts attention to
subsets of size at most $r+1$, a natural
question is whether an extremal separating
atom must actually use the full support
allowed by this bound. This was
formulated by Schefler,
Zhao, and Zhong \cite{SZZ2509} as follows.

\medskip
\begin{conjecture}[{\cite[Conjecture 5.4]{SZZ2509}}]
\label{conj:5.4}
Let $G = C_{n_1} \oplus \dots \oplus C_{n_r}$ with $1 < n_1 \mid \dots \mid n_r$. Let $A$ be a separating atom over $G_0$ with $|A| = \beta_{\mathrm{sep}}(G)$, where $G_0 \subseteq G$ is a subset with $|G_0| \le r + 1$. Then $|\operatorname{supp}(A)| = |G_0| = r + 1$.
\end{conjecture}

\begin{remark}
{\rm
For $r=1$, the group is cyclic, say $G \cong C_n$, and we have $\beta_{\mathrm{sep}}(G) = n$. If we take a subset $G_0 = \{g\}$ consisting of a single generator $g$ of $G$, then the sequence $A = g^n$ is a separating atom over $G_0$ of maximal length $|A| = \beta_{\mathrm{sep}}(G)$. However, in this case,
\[
|\operatorname{supp}(A)| = |G_0| = 1 \neq r+1 = 2.
\]
Thus, the conclusion fails for $r=1$, making the condition $r \ge 2$ indispensable.
}
\end{remark}

\medskip
The primary objective of this paper is to
determine the separating Noether number
for every finite abelian group in closed form.
Our main theorem achieves this by distinguishing
two cases based on the parity of the rank.

\begin{theorem}
\label{formulas}
Let $G\cong C_{n_1}\oplus \cdots \oplus C_{n_r}$ with $1<n_1\mid \cdots \mid n_r$, and let $\beta_{\mathrm{sep}}(G)$ be the separating Noether number of $G$.
\begin{enumerate}
    \item[(1)] If $r=2s-1$, then
    \[
    \beta_{\mathrm{sep}}(G)=n_s+n_{s+1}+\cdots+n_r.
    \]
    \item[(2)] If $r=2s$, then
    \[
    \beta_{\mathrm{sep}}(G)=\frac{n_s}{p_1}+n_{s+1}+\cdots+n_r,
    \]
    where $p_1$ is the minimal prime divisor of $n_1$.
\end{enumerate}
\end{theorem}
As an application of Theorem~\ref{formulas} and the structural rigidity of maximal separating atoms, we prove the conclusion conjectured in \cite[Conjecture~5.4]{SZZ2509} for all finite abelian groups of rank at least $2$.

\begin{corollary}
\label{cor:conj54-resolved}
Let $G\cong C_{n_1}\oplus \cdots
\oplus C_{n_r}$ with $1<n_1\mid \cdots \mid n_r$
and $r\ge 2$. Let $A$ be a separating atom over
a subset $G_0\subseteq G$ such that
$|A|=\beta_{\mathrm{sep}}(G)$ and $|G_0|\le r+1$.
Then
\[
|\operatorname{supp}(A)|=|G_0|=r+1.
\]
In particular, the conclusion of \cite[Conjecture~5.4]{SZZ2509} holds for all finite abelian groups of rank at least $2$.
\end{corollary}

To establish these results, our approach
conceptually departs from the previous literature
at a crucial technical juncture: it avoids the
Davenport-equality hypothesis relied upon
in \cite{SZZ,SZZ2509}. The core of the proof
separates according to the parity of the rank.
For odd rank, we show that the relevant
auxiliary zero-sum free sequences are positive
geodesics, ensuring their lengths are
unconditionally bounded by the positive diameter
theorem of Klopsch and Lev \cite{KL}. To tackle
the more delicate even-rank case, we reduce the
auxiliary sequences to \emph{geodesic surrogates}.
We then combine this geometric
reduction with a uniform short-generation
lemma for relation groups
to lift divisible relations
without any appeal to Davenport equalities.
Together, these ingredients yield the exact
formulas in both parity cases.

The paper is organized as follows.
Section~\ref{pre} collects the necessary results from
additive combinatorics and invariant theory.
Section~\ref{proofs} contains the proofs of
the main theorem and the corollary.

\section{Preliminaries}
\label{pre}
Throughout this paper,
$G\cong C_{n_1}\oplus \cdots \oplus C_{n_r}$
denotes a finite abelian group with
$1<n_1\mid \cdots \mid n_r$
and $r = \mathsf{rank}(G)$.
As usual, let $\mathbb{N}$ denote
the set of positive integers,
$\mathbb{N}_0 = \mathbb{N} \cup \{0\}$,
and $\mathbb{Z}$ the set of integers.
For any two integers $a, b \in \mathbb{Z}$,
we denote the discrete interval between
them by $[a, b] = \{x \in \mathbb{Z}\,|\,a \le x \le b\}$.

For a non-empty subset $G_0 \subseteq G$,
a finite unordered sequence with terms from $G_0$
is viewed as an element of the free abelian
monoid $\mathcal{F}(G_0)$ with basis $G_0$,
written multiplicatively with concatenation
as the monoid operation.
A sequence $S \in \mathcal{F}(G_0)$
can be uniquely expressed in the exponential form
\[
S = g_1 \cdot \ldots \cdot g_\ell = \prod_{g \in G_0} g^{\mathsf{v}_g(S)},
\]
where $\mathsf{v}_g(S) \in \mathbb{N}_0$ denotes
the multiplicity of $g$ in $S$.
The length of the sequence is
$|S| = \ell = \sum_{g \in G_0} \mathsf{v}_g(S)$.
We say that $T \in \mathcal{F}(G_0)$ is a subsequence
of $S$ if $T$ divides $S$ in $\mathcal{F}(G_0)$
(denoted by $T \mid S$); that is,
$\mathsf{v}_g(T)\leq \mathsf{v}_g(S)$
for all $g\in G_0$.
The sequence obtained by deleting the terms
of $T$ from $S$ is denoted by $T^{-1}S$.

For a sequence $S = g_1 \cdot \ldots \cdot g_\ell$,
its sum is defined as $\sigma(S) = \sum_{i=1}^{\ell} g_i$.
The set of all possible sums of its non-empty subsequences is denoted by
\[
\Sigma(S) = \bigg\{ \sum_{i \in I} g_i
\,\Big{|}\,\emptyset \neq I \subseteq [1, |S|] \bigg\}.
\]
Based on these arithmetic properties, we classify sequences over $G_0$ as follows:
\begin{itemize}
    \item $S$ is called a \emph{zero-sum free sequence} if $0 \notin \Sigma(S)$.
    \item $S$ is called a \emph{zero-sum sequence} if $\sigma(S) = 0$. The collection of all such sequences forms a submonoid of $\mathcal{F}(G_0)$, denoted by
    \[
    \mathcal{B}(G_0) = \{S \in \mathcal{F}(G_0)\,|\,\sigma(S) = 0\}.
    \]
    \item $S$ is called an \emph{atom} (or a minimal zero-sum sequence) if it is a nontrivial zero-sum sequence that contains no proper nontrivial zero-sum subsequences. The set of all atoms over $G_0$ is denoted by $\mathcal{A}(G_0)$.
\end{itemize}

The \emph{Davenport constant} of a subset $G_0$, denoted by $\mathsf{D}(G_0)$, is the maximal length of an atom over $G_0$:
\[
\mathsf{D}(G_0) = \max \{ |S|\,|\,S \in \mathcal{A}(G_0) \}.
\]
It is a classic result that for $|G| \ge 2$, the Davenport constant is bounded below by the structural invariant:
\[
\mathsf{D}(G) \ge \mathsf{D}^*(G) := 1 + \sum_{i=1}^r (n_i - 1).
\]

To connect the arithmetic of zero-sum sequences
to the invariant theory of finite groups,
we consider the algebraic structure of
the zero-sum monoid. For $d \in \mathbb{N}_0$, let
\[
\mathcal{B}(G_0)_d= \{S \in \mathcal{B}(G_0)\,|\,|S| \le d\}
\]
be the subset of zero-sum sequences of length at most $d$. Let $q(\mathcal{B}(G_0))$ denote the quotient group (Grothendieck group) of the commutative cancellative monoid $\mathcal{B}(G_0)$. We write $q(\mathcal{B}(G_0)_d)$ for the subgroup of $q(\mathcal{B}(G_0))$ generated by the elements of $\mathcal{B}(G_0)_d$.

\begin{remark}
\label{rem:qBgd}
{\rm
Our notation differs slightly from that of \cite{SZZ}, where $\mathcal{B}(G_0)_d$ denotes the submonoid generated by the atoms of length at most $d$. The resulting subgroup of $q(\mathcal{B}(G_0))$ is the same in both conventions, because every zero-sum sequence of length at most $d$ factors into atoms of length at most $d$, and conversely, every such atom is itself a zero-sum sequence of length at most $d$.
}
\end{remark}

With this algebraic framework, we define the core combinatorial object of our study. An atom $A \in \mathcal{A}(G_0)$ is called a \emph{separating atom} over $G_0$ if it cannot be generated by strictly shorter zero-sum sequences in the quotient group:
$
A \notin q(\mathcal{B}(G_0)_{|A|-1}).
$
In other words, there do not exist atoms $U_1, \dots, U_k, V_1, \dots, V_l \in \mathcal{A}(G_0)$ with lengths $|U_i|, |V_j| \le |A| - 1$ such that
\[
A \cdot U_1 \cdot \ldots \cdot U_k = V_1 \cdot \ldots \cdot V_l.
\]

The study of the separating Noether number $\beta_{\mathrm{sep}}(G)$ is fundamentally reduced to bounding the lengths of these separating atoms. The following foundational lemma, due to Domokos \cite{Domokos17}, ensures that we only need to test small support sets $G_0$.

\begin{lemma}[{\cite[Lemma 2.3]{SZZ}}]
\label{reduced}
Let $G$ be a finite abelian group. Then $\beta_{\mathrm{sep}}(G)$ is the maximal length of a separating atom over $G_0$, where $G_0 \subseteq G$ ranges over all subsets of size $|G_0| \le \mathsf{rank}(G) + 1$.
\end{lemma}

To effectively bound the lengths of auxiliary zero-sum
free sequences without relying on the Davenport equality $\mathsf{D}(G)=\mathsf{D}^*(G)$, we employ
a geometric perspective. We interpret sequences over a subset
$A\subseteq G$
as positive walks in the corresponding directed Cayley graph.
Let $A \subseteq G$ be a subset.
For $j \in \mathbb{N}$, define the $j$-fold sumset
\[
jA = \{a_1 + \dots + a_j\,|\,a_1, \dots, a_j \in A\},
\]
and set $0A = \{0\}$. For $n \in \mathbb{N}_0$, we denote the accumulated sumset up to length $n$ by
\[
[0, n]A = \bigcup_{j=0}^n jA.
\]
For an element $g \in G$, the \emph{positive length}
of $g$ with respect to $A$ is defined as
the minimal number of terms from $A$ required to sum to $g$:
\[
\ell_A^{+}(g) = \min\{\, n \in \mathbb{N}_0\,|\,g \in [0, n]A \,\},
\]
with the convention that $\ell_A^{+}(g) = \infty$
if $g \notin \langle A \rangle$.
If $\langle A \rangle = G$,
the \emph{positive diameter} of $G$ with respect to $A$ is
\[
\operatorname{diam}_A^{+}(G) =
\max\{\, \ell_A^{+}(g)\,|\,g \in G \,\}.
\]
Taking the supremum over all possible generating sets,
the \emph{absolute positive diameter} of $G$ is defined as
\[
\operatorname{diam}^{+}(G) =
\max\{\, \operatorname{diam}_A^{+}(G)\,|\,
A \subseteq G, \ \langle A \rangle = G \,\}.
\]
If $S$ is a sequence whose terms all lie
in $A \subseteq G$, we say that $S$
is \emph{geodesic with respect to $A$}
if its length perfectly matches the
shortest possible path to its sum, i.e.,
\[
|S| = \ell_A^{+}\bigl(\sigma(S)\bigr).
\]

\begin{remark}
{\rm
For completeness regarding the trivial group, we adopt the standard convention that $\beta_{\mathrm{sep}}(\{0\}) = 1$, $\operatorname{diam}^{+}(\{0\}) = 0$, and $\mathsf{D}^{*}(\{0\}) = 1$.
}
\end{remark}

The fundamental link connecting this geometric diameter back to the structural constants of finite abelian groups is the following theorem by Klopsch and Lev \cite{KL}, which serves as an unconditional upper bound for our geodesic sequences.

\begin{lemma}
\label{lem:positive-diameter}
For every generating set $A \subseteq G$ and every $g \in G$, we have
\[
\ell_A^{+}(g) \le \operatorname{diam}_A^{+}(G) \le \operatorname{diam}^{+}(G) = \mathsf{D}^{*}(G) - 1.
\]
In particular, if $H \le G$ is a subgroup and $\langle A \rangle = H$, then for every $g \in H$,
\[
\ell_A^{+}(g) \le \mathsf{D}^{*}(H) - 1 \le \mathsf{D}^{*}(G) - 1.
\]
\end{lemma}

\begin{proof}
The exact evaluation of the absolute positive diameter, $\operatorname{diam}^{+}(G) = \mathsf{D}^{*}(G) - 1$, is established in \cite[Theorem~2.1]{KL}. The remaining inequalities follow immediately from the definitions and the subgroup structure.
\end{proof}

With the geometric tools in place to control the lengths of auxiliary sequences, we need an algebraic mechanism to lift these sequences back to the invariant-theoretic setting of separating atoms. The following lemma, which adapts a crucial insight from \cite[Lemma~3.1]{SZZ}, provides exactly this bridge. It guarantees that any zero-sum sequence whose multiplicities are uniformly divisible by a critical structural parameter $n$ can be automatically generated by strictly shorter zero-sum sequences in the quotient group.

\begin{lemma}[{\cite[Lemma~3.1]{SZZ}}]
\label{lem:n-divisible-lifting}
Let
$
G\cong C_{n_1}\oplus \cdots \oplus C_{n_r}
$
with
$
1<n_1\mid \cdots \mid n_r
$
and
$
r\ge 2.
$
Set
$
s=\Bigl\lfloor \frac{r+1}{2}\Bigr\rfloor
$
and
$
n=n_s.
$
Let $G_0\subseteq G$ be a nonempty subset, and let $A$ be a separating atom over $G_0$ with maximal length
$
|A|=\beta_{\mathrm{sep}}(G).
$
If $S\in \mathcal{B}(G_0)$ is a zero-sum sequence such that
$
n \mid \mathsf{v}_g(S)
$
for every
$
g\in G_0,
$
then $S$ is generated by strictly shorter sequences in the quotient group, meaning
$
S\in q\!\bigl(\mathcal{B}(G_0)_{|A|-1}\bigr).
$
\end{lemma}

\begin{proof}
This is equivalent to \cite[Lemma~3.1]{SZZ} after invoking the maximality $|A|=\beta_{\mathrm{sep}}(G)$ and adopting the subgroup generation convention detailed in Remark~\ref{rem:qBgd}.
\end{proof}

Finally, to complete the setup for our exact formulas, we state the established general lower bound for the separating Noether number. Our primary task in the subsequent sections will be to prove that the lengths of separating atoms are strictly bounded from above by these exact quantities.

\begin{lemma}[{\cite[Lemmas 5.2 and 5.5]{Sch25}}]
\label{lower-bound}
Let $G \cong C_{n_1} \oplus \cdots \oplus C_{n_r}$ with $1 < n_1 \mid \dots \mid n_r$. Let $s = \lfloor \frac{r+1}{2} \rfloor$, and let $p_1$ be the minimal prime divisor of $n_1$. Then the separating Noether number satisfies
\[
\beta_{\mathrm{sep}}(G) \ge
\begin{cases}
n_s + n_{s+1} + \cdots + n_r, & \text{if } r \text{ is odd}, \\
\frac{n_s}{p_1} + n_{s+1} + \cdots + n_r, & \text{if } r \text{ is even}.
\end{cases}
\]
\end{lemma}

\section{Proofs of Theorem \ref{formulas} and
Corollary \ref{cor:conj54-resolved}}
\label{proofs}

In this section, we prove our main results:
the exact formulas for
$\beta_{\mathrm{sep}}(G)$ and the
rank $\ge 2$ case of the conclusion conjectured
in \cite[Conjecture~5.4]{SZZ2509}.
To sharply bound the length of a maximal separating atom $A$,
we decompose it by dividing the multiplicities
of its terms by the critical parameter
$n = n_{\lfloor(r+1)/2\rfloor}$.
This division extracts
a ``quotient sequence'' $B$ over the subgroup $nG$.
The following key lemma demonstrates
that $B$ is a  geodesic  sequence in $nG$,
which intrinsically bounds its length by the
Davenport constant of the subgroup.

\medskip
\begin{lemma}
\label{lem:B-geodesic}
Let
$
G\cong C_{n_1}\oplus \cdots \oplus C_{n_r},
$
where
$
1<n_1\mid \cdots \mid n_r,
$
$
r\ge 2,
$
$s=\lfloor (r+1)/2\rfloor$ and $n=n_s$.
By Lemma \ref{reduced}, choose a subset
$
G_0=\{g_1,\dots,g_k\}\subseteq G
$
with
$
k\le r+1
$
such that
$
A=\prod_{i=1}^k g_i^{m_i}
$
is a separating atom
over $G_0$ satisfying $|A|=\beta_{\mathrm{sep}}(G)$.
Write
$
m_i=nk_i+x_i,
~
k_i\in \N_0,\  x_i\in [0,n-1]
$
for each $i$, and set
\[
h_i=ng_i\in nG,
\qquad
B=\prod_{i=1}^k h_i^{k_i}.
\]
Then $B$ is geodesic in $nG$ with
respect to the set $\{h_1,\dots,h_k\}$:
$
|B|=\ell_{\{h_1,\dots,h_k\}}^{+}\!\bigl(\sigma(B)\bigr).
$
Consequently,
$
|B|\le \mathsf{D}^{*}(nG)-1.
$
\end{lemma}

\begin{proof}
First note that not all $x_i$ are zero.
Indeed, if $x_i=0$ for all $i$,
then $A$ is a zero-sum sequence all of whose
multiplicities are divisible by $n$.
Lemma~\ref{lem:n-divisible-lifting} yields
$
A\in
q\!\bigl(\mathcal{B}(G_0)_{|A|-1}\bigr),
$
contradicting that $A$ is separating.
Hence
$
A^{(1)}=\prod_{i=1}^k g_i^{nk_i}
$
is a proper subsequence of the atom $A$,
and $A^{(1)}$ is zero-sum free in $G$.
We claim that
$
B=\prod_{i=1}^k h_i^{k_i}
$
is zero-sum free in $nG$. Indeed, if
$
T=\prod_{i=1}^k h_i^{u_i}
$
were a nontrivial zero-sum subsequence of $B$,
where $0\le u_i\le k_i$ for all $i$, then
$
\sum_{i=1}^k u_i h_i=0
$
would imply
$
\sum_{i=1}^k (nu_i)g_i=0.
$
It follows that
$
\prod_{i=1}^k g_i^{nu_i}
$
would be a nontrivial zero-sum subsequence
of $A^{(1)}$, a contradiction.
Therefore $B$ is zero-sum free in $nG$.
Assume, for contradiction, that $B$ is not geodesic.
Then there exist $u_1,\dots,u_k\in \N_0$ such that
\[
\sum_{i=1}^k u_i h_i=\sum_{i=1}^k k_i h_i
\qquad\text{and}\qquad
\sum_{i=1}^k u_i < \sum_{i=1}^k k_i.
\]
Define
$
\widetilde{A}=\prod_{i=1}^k g_i^{nu_i+x_i}.
$
Since
\[
\sum_{i=1}^k (nu_i+x_i)g_i
=
\sum_{i=1}^k u_i h_i+\sum_{i=1}^k x_i g_i
=
\sum_{i=1}^k k_i h_i+\sum_{i=1}^k x_i g_i
=
\sum_{i=1}^k m_i g_i
=
0,
\]
the sequence $\widetilde{A}$ is zero-sum.
Its length is
\[
|\widetilde{A}|=n\sum_{i=1}^k u_i+\sum_{i=1}^k x_i
<
n\sum_{i=1}^k k_i+\sum_{i=1}^k x_i
=
|A|,
\]
hence
$
\widetilde{A}\in \mathcal{B}(G_0)_{|A|-1}.
$
Now define
\[
p_i=(k_i-u_i)^{+}=\max\{k_i-u_i,0\},
\qquad
q_i=(u_i-k_i)^{+}=\max\{u_i-k_i,0\},
\]
and define the $n$-divisible sequences
\[
P=\prod_{i=1}^k g_i^{np_i},
\qquad
Q=\prod_{i=1}^k g_i^{nq_i}.
\]
Since
$
k_i+q_i=u_i+p_i
~\text{for all }i,
$
we have
\[
\sum_{i=1}^k p_i h_i=\sum_{i=1}^k q_i h_i=:s\in nG.
\]
Let $L=\langle h_1,\dots,h_k\rangle\le nG$.
Then there exist $c_1,\dots,c_k\in \N_0$ such that
$
\sum_{i=1}^k c_i h_i=-s.
$
Set
$
C=\prod_{i=1}^k g_i^{nc_i}.
$
Then both $PC$ and $QC$ are zero-sum sequences over $G_0$, and all their multiplicities are divisible by $n$.
Lemma~\ref{lem:n-divisible-lifting} therefore gives
$
PC,\ QC \in q\!\bigl(\mathcal{B}(G_0)_{|A|-1}\bigr).
$
Finally, using $k_i+q_i=u_i+p_i$,
we get the exact identity of sequences
$
A\cdot Q\cdot C = \widetilde{A}\cdot P\cdot C.
$
Hence, in the quotient group $q(\mathcal{B}(G_0))$,
\[
A
=
\widetilde{A}\cdot (PC)\cdot (QC)^{-1}
\in
q\!\bigl(\mathcal{B}(G_0)_{|A|-1}\bigr),
\]
because $\widetilde{A}\in \mathcal{B}(G_0)_{|A|-1}$ and $PC,QC\in q(\mathcal{B}(G_0)_{|A|-1})$.
This contradicts the fact that $A$ is a separating atom.
Thus $B$ is geodesic.
Applying Lemma~\ref{lem:positive-diameter} to the subgroup $L=\langle h_1,\dots,h_k\rangle\le nG$ yields
\[
|B|=\ell_{\{h_1,\dots,h_k\}}^{+}\!\bigl(\sigma(B)\bigr)
\le \mathsf{D}^{*}(L)-1
\le \mathsf{D}^{*}(nG)-1.
\]
\end{proof}

\medskip
To extract deeper structural constraints
on the maximal separating atom $A$,
we generalize the previous decomposition
by scaling its multiplicities by an arbitrary
integer $l \in \N$.  Dividing the scaled multiplicities
$l m_i$ by $n$ yields
a new set of remainders $x_i^{(l)}$.
To balance these remainders and form a zero-sum,
we construct a minimal compensating
sequence $Y^{(l)}$ in the subgroup $nG$.
The minimality condition intrinsically
forces $Y^{(l)}$ to be a geodesic sequence,
ensuring its length is strictly
bounded by $\mathsf{D}^*(nG) - 1$.

\medskip
\begin{lemma}
\label{lem:Y-geodesic}
Keep the notation of Lemma~\ref{lem:B-geodesic}. For each $l\in \N$, write
\[
l m_i = n k_i^{(l)} + x_i^{(l)},
\quad
k_i^{(l)}\in \N_0,
\quad
x_i^{(l)}\in [0,n-1]
\quad \hbox{for}~1\le i\le k.
\]
Choose $(t_1^{(l)},\dots,t_k^{(l)})\in \N^k$
with $\sum_{i=1}^kt_i^{(l)}$ minimal subject to
$
\sum_{i=1}^k \bigl(t_i^{(l)} n-x_i^{(l)}\bigr)g_i=0,
$
and define
\[
Y^{(l)}=\prod_{i=1}^k h_i^{\,t_i^{(l)}-1}
=
\prod_{i=1}^k (ng_i)^{\,t_i^{(l)}-1}.
\]
Then $Y^{(l)}$ is geodesic in $nG$ with respect to the
 set $\{h_1,\dots,h_k\}$.
Consequently,
$
|Y^{(l)}|\le \mathsf{D}^{*}(nG)-1.
$
\end{lemma}

\begin{proof}
By Lemma~\ref{lower-bound}, we have
\[
|A|=\beta_{\mathrm{sep}}(G)>n_r.
\]
Indeed, if $r=2s-1$ is odd, then
\[
|A|\ge n_s+n_{s+1}+\cdots+n_r\ge n_s+n_r>n_r,
\]
while if $r=2s$ is even, then
\[
|A|\ge \frac{n_s}{p_1}+n_{s+1}+\cdots+n_r\ge 1+n_r>n_r.
\]
We claim that
$
m_i\le \operatorname{ord}(g_i)-1
$
for $1\le i\le k$.
Indeed, if $m_i\ge \operatorname{ord}(g_i)$ for some $i$,
then $g_i^{\operatorname{ord}(g_i)}$ is a nontrivial zero-sum
subsequence of $A$. Since $A$ is an atom, this forces
$
A=g_i^{\operatorname{ord}(g_i)},
$
and hence
$
|A|=\operatorname{ord}(g_i)\le n_r,
$
contradicting $|A|>n_r$. Therefore
$m_i\le \operatorname{ord}(g_i)-1$ for all $i$.
Hence
$
l m_i \le l\,\operatorname{ord}(g_i)-l,
$
and
$
k_i^{(l)}\le l\,\operatorname{ord}(g_i)-1
$
for
$1\le i\le k.
$
Therefore the tuple
\[
\bigl(l\,\operatorname{ord}(g_i)-k_i^{(l)}\bigr)_{i=1}^k \in \N^k
\]
is well defined and satisfies
\[
\sum_{i=1}^k
\Bigl(\bigl(l\,\operatorname{ord}(g_i)-k_i^{(l)}\bigr)n-x_i^{(l)}\Bigr)g_i
=
\sum_{i=1}^k \bigl(l\,\operatorname{ord}(g_i)n-lm_i\bigr)g_i
=
-l\sum_{i=1}^k m_i g_i
=
0.
\]
Thus the feasible set is nonempty.
Set
$
z^{(l)}=\sigma(Y^{(l)})=\sum_{i=1}^k (t_i^{(l)}-1)h_i \in nG.
$
Suppose there exist $u_1,\dots,u_k\in \N_0$ such that
\[
\sum_{i=1}^k u_i h_i = z^{(l)}
\qquad\text{and}\qquad
\sum_{i=1}^k u_i < \sum_{i=1}^k (t_i^{(l)}-1).
\]
Then
\[
\sum_{i=1}^k \bigl((u_i+1)n-x_i^{(l)}\bigr)g_i
=
\sum_{i=1}^k u_i h_i+\sum_{i=1}^k (n-x_i^{(l)})g_i
=
z^{(l)}+\sum_{i=1}^k (n-x_i^{(l)})g_i
=
0.
\]
Hence $(u_i+1)_{i=1}^k$ is another feasible tuple for
the defining minimization of
$(t_i^{(l)})$, but
$
\sum_{i=1}^k (u_i+1)<\sum_{i=1}^k t_i^{(l)},
$
contradicting minimality.
Thus $Y^{(l)}$ is geodesic.
Applying Lemma~\ref{lem:positive-diameter} to the subgroup
$L=\langle h_1,\dots,h_k\rangle\le nG$ yields
\[
|Y^{(l)}|
=\ell^{+}_{\{h_1,\dots,h_k\}}\!\bigl(\sigma(Y^{(l)})\bigr)
\le \mathsf{D}^{*}(L)-1
\le \mathsf{D}^{*}(nG)-1.
\]
\end{proof}

\medskip
The lower bounds for Theorem \ref{formulas} were
already presented in Lemma \ref{lower-bound}.
It remains to prove the matching upper bounds.
We divide the proof into two parts based on the parity of the rank.
We first resolve the odd-rank case.

\medskip
\begin{proposition}
\label{thm:odd-unconditional}
Let
$
G\cong C_{n_1}\oplus \cdots \oplus C_{n_r}$
with
$
1<n_1\mid \cdots \mid n_r$
and
$
r=2s-1.
$
Then
$
\beta_{\mathrm{sep}}(G)=n_s+n_{s+1}+\cdots+n_r.
$
\end{proposition}

\begin{proof}
If $r=1$, then $G$ is cyclic and
$
\beta_{\mathrm{sep}}(G)=|G|=n_1
$
by \cite[Theorem~3.10]{Domokos17}.
Hence the assertion holds. From now on,
 assume $r\ge 3$.
Set
$
n=n_s.
$
If $nG=\{0\}$, then by convention $\mathsf{D}^{*}(nG)=1$.
In this case,  necessarily $n_{s+1}=\cdots=n_r=n$, and therefore
\[
n\bigl(\mathsf{D}^{*}(nG)-1\bigr)=0=\big(\sum_{j=s+1}^r n_j\big)-(s-1)n.
\]
If $nG\neq \{0\}$, then
$
nG\cong C_{n_{s+1}/n}\oplus \cdots \oplus C_{n_r/n},
$
and hence
\[
n\bigl(\mathsf{D}^{*}(nG)-1\bigr)
=
\big(\sum_{j=s+1}^r n_j\big)-(s-1)n.
\]
For brevity, write
$
M=n\bigl(\mathsf{D}^{*}(nG)-1\bigr)
=
\big(\sum_{j=s+1}^r n_j\big)-(s-1)n.
$
By Lemma \ref{reduced}, choose a separating atom
$
A=\prod_{i=1}^k g_i^{m_i}
$
over some
$
G_0=\{g_1,\dots,g_k\}\subseteq G
$
with
$
k\le r+1=2s
$
such that
$
|A|=\beta_{\mathrm{sep}}(G).
$
Write
$
m_i=nk_i+x_i$
with
$
0\le x_i\le n-1
$
and
define
\[
h_i=ng_i,
\qquad
B=\prod_{i=1}^k h_i^{k_i}.
\]
By Lemma~\ref{lem:B-geodesic},
$
|B|\le \mathsf{D}^{*}(nG)-1,
$
and hence
$
n|B|\le M.
$
Since
$
|A|=n|B|+\sum_{i=1}^k x_i,
$
we obtain
\begin{equation}
\label{eq:A-upper-1}
|A|\le M+\sum_{i=1}^k x_i.
\end{equation}
Let $(t_1,\dots,t_k)\in \N^k$ be chosen as in
Lemma~\ref{lem:Y-geodesic} for $l=1$; that is,
$
\sum_{i=1}^k (t_i n-x_i)g_i=0.
$
Set
\[
V=\prod_{i=1}^k g_i^{t_i n-x_i},
\qquad
Y=\prod_{i=1}^k h_i^{t_i-1}.
\]
By Lemma~\ref{lem:Y-geodesic},
$
|Y|\le \mathsf{D}^{*}(nG)-1,
$
giving
$
n|Y|\le M.
$
Next we claim that
$
|V|\ge |A|.
$
Indeed, the sequence
$
AV=\prod_{i=1}^k g_i^{(k_i+t_i)n}
$
is a zero-sum sequence all of whose multiplicities are divisible by $n$.
Since $|A|=\beta_{\mathrm{sep}}(G)$,
Lemma~\ref{lem:n-divisible-lifting}
 gives
$
AV \in q\!\bigl(\mathcal{B}(G_0)_{|A|-1}\bigr).
$
If $|V|<|A|$, then $V\in \mathcal{B}(G_0)_{|A|-1}$, and therefore
$
A=(AV)\cdot V^{-1}\in q\!\bigl(\mathcal{B}(G_0)_{|A|-1}\bigr),
$
contradicting that $A$ is separating.
Thus $|V|\ge |A|$.
On the other hand,
\[
|V|
=
\sum_{i=1}^k (t_i n-x_i)
=
n\sum_{i=1}^k (t_i-1)+kn-\sum_{i=1}^k x_i
=
n|Y|+kn-\sum_{i=1}^k x_i
\le
M+kn-\sum_{i=1}^k x_i.
\]
Since $k\le 2s$, we get
\begin{equation}
\label{eq:A-upper-2}
|A|
\le
|V|
\le
M+2sn-\sum_{i=1}^k x_i.
\end{equation}

Adding \eqref{eq:A-upper-1} and \eqref{eq:A-upper-2}, we obtain
\[
2|A|
\le
2M+2sn
=
2\sum_{j=s+1}^r n_j - 2(s-1)n + 2sn
=
2\sum_{j=s+1}^r n_j + 2n.
\]
Hence
$
|A|\le n+\sum_{j=s+1}^r n_j
=
\sum_{j=s}^r n_j.
$
Since $|A|=\beta_{\mathrm{sep}}(G)$, this proves
$
\beta_{\mathrm{sep}}(G)\le \sum_{j=s}^r n_j.
$
The opposite inequality
$
\beta_{\mathrm{sep}}(G)\ge \sum_{j=s}^r n_j
$
is exactly Lemma \ref{lower-bound} in the odd-rank case.
Therefore
$
\beta_{\mathrm{sep}}(G)=n_s+n_{s+1}+\cdots+n_r.
$
\end{proof}

\medskip
We now turn to the more intricate even-rank case,
where $r=2s$. Unlike the odd-rank case,
the exact formula for even rank depends on
the minimal prime divisor $p_1$ of $n_1$.
As a crucial stepping stone towards the
final exact bound, we first establish a
slightly weaker intermediate upper bound
involving $n_s/2$. This ``half-bound''
imposes a strict structural constraint
that will be instrumental in our final deduction.

\medskip
\begin{proposition}
\label{thm:even-half-bound}
Let
$
G\cong C_{n_1}\oplus \cdots \oplus C_{n_r}
$
with
$
1<n_1\mid \cdots \mid n_r
$
and
$
r=2s.
$
Then
$
\beta_{\mathrm{sep}}(G)\le \frac{n_s}{2}+n_{s+1}+\cdots+n_r.
$
\end{proposition}

\begin{proof}
Let
$
n=n_s$
and
$
T=\sum_{j=s+1}^r n_j.
$
If $nG=\{0\}$, then $n_{s+1}=\cdots=n_r=n$,
so by convention $\mathsf{D}^*(nG)=1$ and
$
n\bigl(\mathsf{D}^*(nG)-1\bigr)=0=T-sn.
$
If $nG\neq \{0\}$, then, after omitting trivial factors,
\[
nG\cong C_{n_{s+1}/n}\oplus\cdots\oplus C_{n_r/n},
\]
and therefore
\[
n\bigl(\mathsf{D}^*(nG)-1\bigr)
=
n\sum_{j=s+1}^r\Bigl(\frac{n_j}{n}-1\Bigr)
=
T-sn.
\]
Set
$
M=n\bigl(\mathsf{D}^*(nG)-1\bigr)=T-sn.
$
By Lemma \ref{reduced}, choose a subset
\[
G_0=\{g_1,\dots,g_k\}\subseteq G,
\qquad
k\le r+1=2s+1,
\]
and a separating atom
$
A=\prod_{i=1}^k g_i^{m_i}
$
over $G_0$ such that
$
|A|=\beta_{\mathrm{sep}}(G).
$
Write
$
m_i=nk_i+x_i
$
with
$
0\le x_i\le n-1,
$
and set
$
h_i=ng_i
$
and
$
B=\prod_{i=1}^k h_i^{k_i}.
$
By Lemma~\ref{lem:B-geodesic},
$
|B|\le \mathsf{D}^*(nG)-1,
$
hence
\begin{equation}\label{E1}
|A|=n|B|+\sum_{i=1}^k x_i\le M+\sum_{i=1}^k x_i.
\end{equation}
Now let $(t_1,\dots,t_k)\in \N^k$ be chosen as in
Lemma~\ref{lem:Y-geodesic} for $l=1$; that is,
$
\sum_{i=1}^k (t_i n-x_i)g_i=0
$
and $\sum_{i=1}^k t_i$ is minimal among all
tuples in $\N^k$ with this property.
Define
\[
V=\prod_{i=1}^k g_i^{t_i n-x_i},
\qquad
Y=\prod_{i=1}^k h_i^{t_i-1}.
\]
As in the proof of
Proposition~\ref{thm:odd-unconditional}, the sequence
$
AV=\prod_{i=1}^k g_i^{(k_i+t_i)n}
$
is a zero-sum sequence over $G_0$,
and all its multiplicities are divisible by $n$.
Hence Lemma~\ref{lem:n-divisible-lifting} gives
$
AV\in q\!\bigl(\mathcal{B}(G_0)_{|A|-1}\bigr).
$
If we had $|V|\le |A|-1$, then
$V\in \mathcal{B}(G_0)_{|A|-1}\subseteq q(\mathcal{B}(G_0)_{|A|-1})$
and
$
A=(AV)\cdot V^{-1}\in q\!\bigl(\mathcal{B}(G_0)_{|A|-1}\bigr),
$
contradicting that $A$ is a separating atom.
Hence
$
|V|\ge |A|.
$
By Lemma~\ref{lem:Y-geodesic},
$
|Y|\le \mathsf{D}^*(nG)-1,
$
and therefore
\[
|V|
=
n|Y|+kn-\sum_{i=1}^k x_i
\le
M+kn-\sum_{i=1}^k x_i.
\]
Since $|V|\ge |A|$, we get
\begin{equation}\label{E2}
|A|\le M+kn-\sum_{i=1}^k x_i.
\end{equation}
Adding \eqref{E1} and \eqref{E2} yields
$
2|A|
\le
2M+kn
\le
2(T-sn)+(2s+1)n
=
2T+n.
$
Hence
\[
|A|\le T+\frac{n}{2}
=
\frac{n_s}{2}+n_{s+1}+\cdots+n_r.
\]
Since $|A|=\beta_{\mathrm{sep}}(G)$, the proof is complete.
\end{proof}

\medskip
Before proving the exact formula involving
the minimal prime divisor $p_1$, we establish a key structural inequality connecting the separating Noether number of the subgroup $n_1G$ to the invariants of $G$. This rank-shift bound will serve as the final ingredient for the even-rank case.

\medskip
\begin{lemma}
\label{lem:n1-rank-shift}
Let
$
G\cong C_{n_1}\oplus \cdots \oplus C_{n_r}
$
with
$
1<n_1\mid \cdots \mid n_r
$
and
$
r=2s.
$
Let
$
a=\max\{j\in [1,r]\,|\,n_j=n_1\}.
$
Assume $n_1G$ is nontrivial, and write
\[
n_1G\cong C_{m_1}\oplus \cdots \oplus C_{m_t},
\qquad
m_j=\frac{n_{a+j}}{n_1},
\qquad
t=r-a.
\]
Then
$
n_1\,\beta_{\mathrm{sep}}(n_1G)\le \sum_{j=s+1}^r n_j.
$
\end{lemma}

\begin{proof}
Set
$
T=\sum_{j=s+1}^r n_j.
$
We consider two cases separately.

\smallskip
\noindent
\textbf{Case 1}: $t$ is odd.
Write $t=2u-1$.
By Proposition~\ref{thm:odd-unconditional},
$
\beta_{\mathrm{sep}}(n_1G)=m_u+\cdots+m_t.
$
Hence
$
n_1\beta_{\mathrm{sep}}(n_1G)=n_{a+u}+\cdots+n_r.
$
Since $t=r-a=2s-a$ is odd, the integer $a$ is odd, and thus
\[
u=\frac{t+1}{2}=s-\frac{a-1}{2},
\qquad
a+u=s+\frac{a+1}{2}\ge s+1.
\]
Therefore
\[
n_1\beta_{\mathrm{sep}}(n_1G)=n_{a+u}+\cdots+n_r\le n_{s+1}+\cdots+n_r=T.
\]

\smallskip
\noindent
\textbf{Case 2}: $t$ is even.
Write $t=2u$.
By Proposition~\ref{thm:even-half-bound},
$
\beta_{\mathrm{sep}}(n_1G)\le \frac{m_u}{2}+m_{u+1}+\cdots+m_t.
$
Multiplying by $n_1$ gives
$
n_1\beta_{\mathrm{sep}}(n_1G)
\le
\frac{n_{a+u}}{2}+n_{a+u+1}+\cdots+n_r.
$
Since $t=2s-a$ is even, the integer $a$ is even, hence $a\ge 2$ and
\[
u=\frac{t}{2}=s-\frac{a}{2},
\qquad
a+u=s+\frac{a}{2}\ge s+1.
\]
Thus
\[
\frac{n_{a+u}}{2}+n_{a+u+1}+\cdots+n_r
\le
n_{s+1}+\cdots+n_r=T.
\]
So again
$
n_1\beta_{\mathrm{sep}}(n_1G)\le T.
$
This completes the proof.
\end{proof}

\medskip
To construct the necessary auxiliary
sequences for our subsequent arguments,
we first recall and fix the core notation
for the even-rank case.
Let $G\cong C_{n_1}\oplus \cdots \oplus C_{n_r}$
with $1<n_1\mid \cdots \mid n_r$ and $r=2s$.
Write $n=n_s$ and $T=\sum_{j=s+1}^r n_j$.
Let $G_0=\{g_1,\dots,g_k\}\subseteq G$ with $k\le 2s+1$,
and let $A=\prod_{i=1}^k g_i^{m_i} \in \mathcal{F}(G_0)$
be a separating atom over $G_0$ such that $|A|=\beta_{\mathrm{sep}}(G)$.
For each positive integer $l \in \mathbb{N}$,
performing division with remainder by $n$ yields
\[
lm_i = n k_i^{(l)} + x_i^{(l)},
\qquad
k_i^{(l)}\in \mathbb{N}_0,
\quad
x_i^{(l)}\in [0,n-1]
\qquad (1 \le i \le k).
\]
We define $h_i=ng_i\in nG$ and the quotient sequence
$B^{(l)}=\prod_{i=1}^k h_i^{k_i^{(l)}}\in \mathcal{F}(nG)$.
We also denote the sum of the remainders by
$f(l)=\sum_{i=1}^k x_i^{(l)}$.

\medskip
\begin{definition}
\label{def:geodesic-surrogate}
For each $l\in\mathbb{N}$, choose coefficients $u_1^{(l)},\dots,u_k^{(l)}\in \mathbb{N}_0$ that minimize the sum $\sum_{i=1}^k u_i^{(l)}$ subject to the constraint
\[
\sum_{i=1}^k u_i^{(l)}h_i = \sum_{i=1}^k k_i^{(l)}h_i.
\]
This minimal sum precisely equals the positive length $\ell^{+}_{\{h_1,\dots,h_k\}}\!\bigl(\sigma(B^{(l)})\bigr)$.
We define the  geodesic surrogate
$U^{(l)} \in \mathcal{F}(nG)$ and
its lifted sequence $\widehat{W}^{(l)} \in \mathcal{F}(G_0)$ as
\[
U^{(l)}=\prod_{i=1}^k h_i^{u_i^{(l)}},
\qquad
\widehat{W}^{(l)}=\prod_{i=1}^k g_i^{n u_i^{(l)}+x_i^{(l)}}.
\]
\end{definition}

The construction of the geodesic surrogate
immediately yields two fundamental properties of the lifted sequence $\widehat{W}^{(l)}$.
\begin{lemma}
\label{lem:surrogate-basic}
For every $l\in \N$ the following hold.
\begin{enumerate}[label=\rm(\alph*),leftmargin=*]
\item $\widehat{W}^{(l)}$ is a zero-sum sequence over $G_0$.
\item
\[
|\widehat{W}^{(l)}|
\le
n\bigl(\mathsf{D}^{*}(nG)-1\bigr)+f(l)
=
T-sn+f(l).
\]
\end{enumerate}
\end{lemma}

\begin{proof}
Since
$
\sum_{i=1}^k u_i^{(l)}h_i=\sum_{i=1}^k k_i^{(l)}h_i,
$
we get
\[
\sigma\!\bigl(\widehat{W}^{(l)}\bigr)
=
\sum_{i=1}^k \bigl(nu_i^{(l)}+x_i^{(l)}\bigr)g_i
=
\sum_{i=1}^k k_i^{(l)}h_i + \sum_{i=1}^k x_i^{(l)}g_i
=
\sum_{i=1}^k l m_i g_i
=
l\,\sigma(A)
=
0,
\]
giving $\widehat{W}^{(l)}\in \mathcal{B}(G_0)$.
For the length bound, let
$
L=\langle h_1,\dots,h_k\rangle \le nG.
$
By Lemma \ref{lem:positive-diameter},
\[
|U^{(l)}|
=
\sum_{i=1}^k u_i^{(l)}
=
\ell^{+}_{\{h_1,\dots,h_k\}}\!\Bigl(\sum_{i=1}^k k_i^{(l)}h_i\Bigr)
\le
\mathsf{D}^{*}(L)-1
\le
\mathsf{D}^{*}(nG)-1.
\]
Hence
$
|\widehat{W}^{(l)}|
=
n|U^{(l)}|+f(l)
\le
n\bigl(\mathsf{D}^{*}(nG)-1\bigr)+f(l).
$
Since
$
n\bigl(\mathsf{D}^{*}(nG)-1\bigr)=\sum_{j=s+1}^r n_j-sn=T-sn,
$
the proof is complete.
\end{proof}

\medskip
Having securely bounded the length of the lifted sequence
 $\widehat{W}^{(l)}$ in terms of $f(l)$,
we now establish a complementary upper
bound for the sequence $V^{(l)}$,
which compensates for the remainders $x_i^{(l)}$
from the opposite direction.

\medskip
\begin{lemma}
\label{lem:V-upper}
For every $l\in \N$,
let $(t_1^{(l)},\dots,t_k^{(l)})\in \N^k$ be chosen with
$
\sum_{i=1}^k (t_i^{(l)}n - x_i^{(l)})g_i=0
$
and with $\sum_{i=1}^kt_i^{(l)}$ minimal, and define
\[
V^{(l)}=\prod_{i=1}^k g_i^{\,t_i^{(l)}n-x_i^{(l)}},
\qquad
Y^{(l)}=\prod_{i=1}^k h_i^{\,t_i^{(l)}-1}.
\]
Then
$
|V^{(l)}|
\le
T+(k-s)n-f(l).
$
In particular,
$
|\widehat{W}^{(l)}|+|V^{(l)}|
\le
2T+(k-2s)n
\le
2T+n.
$
\end{lemma}

\begin{proof}
By Lemma~\ref{lem:Y-geodesic},
$
|Y^{(l)}|
\le
\mathsf{D}^{*}(nG)-1
=
\frac{T}{n}-s.
$
Therefore
$
n\sum_{i=1}^k (t_i^{(l)}-1)
=
|Y^{(l)}|\,n
\le
T-sn,
$
hence
\[
|V^{(l)}|
=
\sum_{i=1}^k (t_i^{(l)}n-x_i^{(l)})
=
n\sum_{i=1}^k (t_i^{(l)}-1)+kn-f(l)
\le
T-sn+kn-f(l).
\]
This gives the first claim.
Adding this to Lemma~\ref{lem:surrogate-basic}(b) yields
\[
|\widehat{W}^{(l)}|+|V^{(l)}|
\le
(T-sn+f(l)) + (T+(k-s)n-f(l))
=
2T+(k-2s)n
\le
2T+n,
\]
since $k\le 2s+1$.
\end{proof}

\medskip
We next show how a short lifted sequence
forces a power of $A$ into
$q(\mathcal{B}(G_0)_{|A|-1})$.

\medskip
\begin{lemma}
\label{lem:geodesic-lift}
Fix $l\in \N$.
If
$
\min\{\,|\widehat{W}^{(l)}|,\ |V^{(l)}|\,\}\le |A|-1,
$
then
$
A^l \in q\!\bigl(\mathcal{B}(G_0)_{|A|-1}\bigr).
$
\end{lemma}

\begin{proof}
Set
\[
p_i^{(l)} = \bigl(k_i^{(l)}-u_i^{(l)}\bigr)^+=\max\{k_i^{(l)}-u_i^{(l)},0\},
~~
q_i^{(l)} = \bigl(u_i^{(l)}-k_i^{(l)}\bigr)^+=\max\{u_i^{(l)}-k_i^{(l)},0\}.
\]
Then
$
k_i^{(l)} + q_i^{(l)} = u_i^{(l)} + p_i^{(l)}
$
for
$1\le i\le k$
and since $\sigma(U^{(l)})=\sigma(B^{(l)})$,
\[
\sum_{i=1}^k p_i^{(l)} h_i
=
\sum_{i=1}^k q_i^{(l)} h_i
=: s^{(l)} \in \langle h_1,\dots,h_k\rangle.
\]
Since
$
-\,s^{(l)}\in \langle h_1,\dots,h_k\rangle,
$
there exist integers
$
c_1^{(l)},\dots,c_k^{(l)}\in \N_0
$
such that
$
\sum_{i=1}^k c_i^{(l)} h_i = -\,s^{(l)}.
$
Define the $n$-divisible sequences
\[
P^{(l)}=\prod_{i=1}^k g_i^{\,n p_i^{(l)}},
\qquad
Q^{(l)}=\prod_{i=1}^k g_i^{\,n q_i^{(l)}},
\qquad
C^{(l)}=\prod_{i=1}^k g_i^{\,n c_i^{(l)}}.
\]
Then
$
P^{(l)}C^{(l)},\ Q^{(l)}C^{(l)} \in \mathcal{B}(G_0),
$
and all their multiplicities are divisible by $n$.
We then have
\begin{equation}
\label{GL1}
P^{(l)}C^{(l)},\ Q^{(l)}C^{(l)}
\in
q\!\bigl(\mathcal{B}(G_0)_{|A|-1}\bigr).
\end{equation}
Now compare exponents:
\[
A^l Q^{(l)} C^{(l)}
=
\prod_{i=1}^k g_i^{\,n(k_i^{(l)}+q_i^{(l)}+c_i^{(l)}) + x_i^{(l)}}
=
\prod_{i=1}^k g_i^{\,n(u_i^{(l)}+p_i^{(l)}+c_i^{(l)}) + x_i^{(l)}}
=
\widehat{W}^{(l)} P^{(l)} C^{(l)}.
\]
Therefore, in $q(\mathcal{B}(G_0))$,
\begin{equation}
\label{GL2}
A^l
=
\widehat{W}^{(l)} \bigl(P^{(l)}C^{(l)}\bigr)\bigl(Q^{(l)}C^{(l)}\bigr)^{-1}.
\end{equation}
Next note that
$
\widehat{W}^{(l)}V^{(l)}
=
\prod_{i=1}^k g_i^{\,n(u_i^{(l)}+t_i^{(l)})}.
$
Thus every multiplicity in $\widehat{W}^{(l)}V^{(l)}$ is divisible by $n$.
Moreover,
$
\sum_{i=1}^k t_i^{(l)}n g_i=\sum_{i=1}^k x_i^{(l)}g_i
$
by the defining relation of $V^{(l)}$, while
\[
\sum_{i=1}^k k_i^{(l)}h_i
=
\sum_{i=1}^k n k_i^{(l)}g_i
=
-\sum_{i=1}^k x_i^{(l)}g_i
\]
because
\[
\sum_{i=1}^k \bigl(nk_i^{(l)}+x_i^{(l)}\bigr)g_i
=
\sum_{i=1}^k l m_i g_i
=
0.
\]
Hence
\[
\sigma\!\bigl(\widehat{W}^{(l)}V^{(l)}\bigr)
=
\sum_{i=1}^k n(u_i^{(l)}+t_i^{(l)})g_i
=
\sum_{i=1}^k k_i^{(l)}h_i+\sum_{i=1}^k t_i^{(l)}n g_i
=
0.
\]
Note that $\widehat{W}^{(l)}V^{(l)}$ is an
$n$-divisible zero-sum sequence, so again
\begin{equation}
\label{GL3}
\widehat{W}^{(l)}V^{(l)} \in q\!\bigl(\mathcal{B}(G_0)_{|A|-1}\bigr).
\end{equation}
If $|\widehat{W}^{(l)}|\le |A|-1$, then
$
\widehat{W}^{(l)}\in \mathcal{B}(G_0)_{|A|-1}.
$
Together with \eqref{GL1} and \eqref{GL2}, this implies
$
A^l \in q\!\bigl(\mathcal{B}(G_0)_{|A|-1}\bigr).
$

If $|V^{(l)}|\le |A|-1$, then
$
V^{(l)}\in \mathcal{B}(G_0)_{|A|-1}.
$
Hence, from \eqref{GL2} and \eqref{GL3}, we get
\[
A^l
=
\bigl(\widehat{W}^{(l)}V^{(l)}\bigr)\,
\bigl(V^{(l)}\bigr)^{-1}\,
\bigl(P^{(l)}C^{(l)}\bigr)\,
\bigl(Q^{(l)}C^{(l)}\bigr)^{-1}
\in
q\!\bigl(\mathcal{B}(G_0)_{|A|-1}\bigr).
\]
This proves the lemma.
\end{proof}

\medskip
The following lemma distills and reformulates
Claim  A4 from \cite{SZZ2509}.

\medskip
\begin{lemma}
\label{lem:V-lower-coprime}
Fix $l\in \N$ with $\gcd(l,n)=1$. Then
$
|V^{(l)}|\ge |A|.
$
\end{lemma}

\begin{proof}
Choose \(l'\in [1,n-1]\) and \(h\in \N_0\) such that
$
ll'=1+hn.
$
The sequence \(A^{hn}\) is zero-sum and all its
multiplicities are divisible by \(n\),
so by Lemma~\ref{lem:n-divisible-lifting},
$
A^{hn}\in q\!\bigl(\mathcal{B}(G_0)_{|A|-1}\bigr).
$
Also
$
A^lV^{(l)}=\prod_{i=1}^k g_i^{\,n(k_i^{(l)}+t_i^{(l)})}
$
is an $n$-divisible zero-sum sequence, hence again
$
A^lV^{(l)}\in q\!\bigl(\mathcal{B}(G_0)_{|A|-1}\bigr).
$
If \(|V^{(l)}|\le |A|-1\),
then \(V^{(l)}\in \mathcal{B}(G_0)_{|A|-1}\), and therefore
\[
A=(A^{hn})^{-1}(V^{(l)})^{-l'}(A^lV^{(l)})^{l'}
\in
q\!\bigl(\mathcal{B}(G_0)_{|A|-1}\bigr),
\]
contradicting that \(A\) is a separating atom. Thus \(|V^{(l)}|\ge |A|\).
\end{proof}

\medskip
The following proposition demonstrates that
if $|A|$ strictly exceeds the target bound,
a coprime power of $A$ is forced into the quotient group.

\medskip
\begin{proposition}
\label{thm:A6-free}
Assume
$
T+\frac{n}{p_1} < |A| \le T+\frac{n}{2},
$
where $p_1$ is the minimal prime divisor of $n_1$.
Then there exists an integer $m\in \N$ such that
\[
\gcd(m,n_1)=1
\qquad\text{and}\qquad
A^m \in q\!\bigl(\mathcal{B}(G_0)_{|A|-1}\bigr).
\]
\end{proposition}

\begin{proof}
Since $n=n_s$ divides $n_j$ for every $j\in\{s+1,\dots,r\}$, we have
$
n\mid T.
$
Write
$
|A| = T+\delta
$
with
$
\frac{n}{p_1}<\delta\le \frac{n}{2}.
$
Let
\[
d=\gcd(|A|,n)=\gcd(\delta,n),
\qquad
\delta = bd
\]
with $b\in \N$ and $\gcd\!\left(b,\frac{n}{d}\right)=1$.

\smallskip
\noindent
\textbf{Case 1: $b\ge 2$.}
By \cite[Lemma~4.2]{Sch25}, there exists
$
l\in \{1,2,\dots,n-1\}
$
such that
\[
\gcd(l,n)=1
\qquad\text{and}\qquad
lb \equiv 1 \pmod{n/d}.
\]
This gives
\begin{equation}
\label{A61}
l\delta = lbd \equiv d \pmod n.
\end{equation}
Since $\gcd(l,n)=1$, Lemma~\ref{lem:V-lower-coprime} gives
$
|V^{(l)}|\ge |A|=T+\delta.
$
Combining this with Lemma~\ref{lem:V-upper},
$
T+\delta
\le
|V^{(l)}|
\le
T+(k-s)n-f(l),
$
hence
\begin{equation}
\label{A62}
f(l)\le (k-s)n-\delta \le (s+1)n-\delta.
\end{equation}
Also, since $x_i^{(l)}\equiv l m_i \pmod n$
 for each $i$, and since $n\mid T$, we get from
\eqref{A61} that
$
f(l)=\sum_{i=1}^k x_i^{(l)}
\equiv
\sum_{i=1}^k l m_i
=
l|A|
=
l(T+\delta)
\equiv
l\delta
\equiv d
\pmod n.
$
Write
$
f(l)=d+qn
~\text{with }q\in\N_0.
$
Then \eqref{A62} implies
$
d+qn\le (s+1)n-\delta.
$
Since $d+\delta>0$, we must have $q\le s$. Therefore
$
f(l)\le sn+d.
$
Now Lemma~\ref{lem:surrogate-basic} gives
$
|\widehat{W}^{(l)}|
\le
T-sn+f(l)
\le
T+d.
$
Since $b\ge 2$,
$
T+d < T+bd = T+\delta = |A|.
$
Therefore
$
|\widehat{W}^{(l)}|\le |A|-1.
$
By Lemma~\ref{lem:geodesic-lift},
$
A^l \in q\!\bigl(\mathcal{B}(G_0)_{|A|-1}\bigr).
$
Finally, $\gcd(l,n)=1$ implies $\gcd(l,n_1)=1$ because $n_1\mid n$.
Thus we may take $m=l$.

\smallskip
\noindent
\textbf{Case 2: $b=1$.}
In this case, we have
\[
\delta=d,
\qquad
|A|=T+d,
\qquad
\frac{n}{p_1}<d\le \frac{n}{2}.
\]
Set
$
m=\frac{n}{d},
$
thus
$
m < p_1
$
and
$
\gcd(m,n_1)=1.
$
Since $x_i^{(m)}\equiv m m_i \pmod n$
for each $i$, and since $n\mid T$, we have
\[
f(m)\equiv m|A| = m(T+\delta)\equiv m\delta = md = n \equiv 0 \pmod n.
\]
Hence $f(m)$ is divisible by $n$.
By Lemma  \ref{lem:V-upper},
\begin{equation}
\label{A66}
|\widehat{W}^{(m)}|+|V^{(m)}|
\le
2T+n.
\end{equation}
Now both $|\widehat{W}^{(m)}|$ and $|V^{(m)}|$ are divisible by $n$:
the first because
$
|\widehat{W}^{(m)}| = n|U^{(m)}|+f(m),
$
and the second because
$
|V^{(m)}| = n\sum_{i=1}^k t_i^{(m)} - f(m).
$
If both were strictly larger than $T$,
then, because $n\mid T$ and both lengths are divisible by $n$,
they would both be at least $T+n$, contradicting \eqref{A66}.
Hence
$
\min\{\,|\widehat{W}^{(m)}|,\ |V^{(m)}|\,\}\le T.
$
Since
$
T< T+d = |A|,
$
we obtain
$
\min\{\,|\widehat{W}^{(m)}|,\ |V^{(m)}|\,\}\le |A|-1.
$
Lemma~\ref{lem:geodesic-lift} now gives
$
A^m \in q\!\bigl(\mathcal{B}(G_0)_{|A|-1}\bigr).
$
This completes the proof.
\end{proof}

\medskip
With Proposition~\ref{thm:A6-free} established,
we have shown that a coprime power of
the atom $A$ is forced into
$q(\mathcal{B}(G_0)_{|A|-1})$.
To translate this membership into an
explicit sequence factorization, we must analyze
the algebraic structure of the underlying
zero-sum relations. We now introduce
the necessary notation and a uniform
short generation lemma to facilitate this final step.
For a finite abelian group $H$ and a sequence
$a_1,\dots,a_k$ of elements of $H$
(not necessarily distinct), write
\[
\mathcal{G}_H(a_1,\dots,a_k)
=
\left\{
\mathbf{u}=(u_1,\dots,u_k)\in \mathbb{Z}^k\,\Big{|}\,
\sum_{i=1}^k u_i a_i=0 \text{ in } H
\right\}.
\]
Also write
$
\mathcal{B}_H(a_1,\dots,a_k)
=
\mathcal{G}_H(a_1,\dots,a_k)\cap \mathbb{N}_0^k.
$
For $\mathbf{u}=(u_1,\dots,u_k)\in \mathbb{N}_0^k$,   define
\[
|\mathbf{u}|=u_1+\cdots+u_k
~\hbox{
and
}~
\supp(\mathbf{u})=\{\,i\in [1,k]\,|\,u_i\ne 0\,\}.
\]
Let $\kappa(H)$ denote the Helly dimension of $H$;
one has
$
\kappa(H)=\mathsf{rank}(H)+1
$
by \cite[\S 2]{Domokos17}.

\medskip
\begin{lemma}
\label{lem:uniform-short-generation}
Let $H$ be a finite abelian group, and set
$
d=\beta_{\mathrm{sep}}(H).
$
Then for every positive integer $k$ and every sequence
$
a_1,\dots,a_k\in H
$
(not necessarily distinct), the abelian group
$
\mathcal{G}_H(a_1,\dots,a_k)
$
is generated by the subset
$
\{\,\mathbf{u}\in \mathcal{B}_H(a_1,\dots,a_k)\,|\,|\mathbf{u}|\le d\,\}.
$
\end{lemma}

\begin{proof}
If $H=\{0\}$, then every $a_i=0$, so
$
\mathcal{G}_H(a_1,\dots,a_k)=\mathbb{Z}^k.
$
This group is generated by the standard basis
vectors $\mathbf{e}_1,\ldots,\mathbf{e}_k$, and each
$\mathbf{e}_i$ lies in $\mathcal{B}_H(a_1,\dots,a_k)$ with
$
|\mathbf{e}_i|=1=\beta_{\mathrm{sep}}(H).
$
Hence the assertion holds.
Assume now that $H\neq \{0\}$, and set
\[
d=\beta_{\mathrm{sep}}(H),
\qquad
M=\{\,\mathbf{u}\in \mathcal{B}_H(a_1,\dots,a_k)\,|\,
|\mathbf{u}|\le d\,\}.
\]
For every subset $J\subseteq [1,k]$, set
$
M_J=\{\,\mathbf{u}\in M\,|\,\supp(\mathbf{u})\subseteq J\,\}.
$
We first claim that for every $J\subseteq [1,k]$ with $|J|\le \kappa(H)$, the group
$
\mathcal{G}_H\bigl((a_j)_{j\in J}\bigr)
$
is generated by $M_J$.
After reindexing we may assume that
$
J=[1,s]
~\text{with}~
s\le \kappa(H).
$
We prove the claim by induction on $s$.
If $a_1,\dots,a_s$ are distinct, then the claim is exactly \cite[Corollary~2.6]{Domokos17},
because $d=\beta_{\mathrm{sep}}(H)$.

Suppose now that the sequence $a_1,\dots,a_s$ is not distinct.
After a permutation of the indices, we may assume that
$
a_1=a_2.
$
Let
$
q=\ord(a_1).
$
Applying \cite[Corollary~2.6]{Domokos17} to
the one-term distinct sequence $(a_1)$ shows that
the group
$
\mathcal{G}_H(a_1)=q\mathbb{Z}
$
is generated by nonnegative relations of length at most $d$.
Since its smallest positive generator is $q$, we must have
$
q\le d.
$
Now define
\[
\mathbf{n}=(1,q-1,0,\dots,0)\in \mathbb{N}_0^s.
\]
Because $a_1=a_2$, we have
$
a_1+(q-1)a_2=qa_1=0,
$
so
$
\mathbf{n}\in \mathcal{B}_H(a_1,\dots,a_s)
$
and
$
|\mathbf{n}|=q\le d.
$
Thus
$
\mathbf{n}\in M_{[1,s]}.
$
Take any
$
\mathbf{u}=(u_1,\dots,u_s)\in \mathcal{G}_H(a_1,\dots,a_s).
$
Set
$
\widetilde{\mathbf{u}}=\mathbf{u}-u_1\mathbf{n}.
$
Then the first coordinate of
$\widetilde{\mathbf{u}}$ is zero, and
$
\sum_{i=1}^s \widetilde u_i a_i
=
\sum_{i=1}^s u_i a_i-u_1qa_1
=
0.
$
Hence
$
\widetilde{\mathbf{u}}\in \mathcal{G}_H(a_1,\dots,a_s)
$
and its first coordinate is zero,
so it is the zero-extension of an element of
$
\mathcal{G}_H(a_2,\dots,a_s).
$
By the induction hypothesis, this element of $\mathcal{G}_H(a_2,\dots,a_s)$ lies in the subgroup generated by
the vectors in
$
\mathcal{B}_H(a_2,\dots,a_s)
$
of length at most $d$.
After adding a zero first coordinate, these generators become elements of
$
M_{[2,s]}\subseteq M_{[1,s]}.
$
Therefore $\widetilde{\mathbf{u}}$ lies in the subgroup generated by $M_{[2,s]}$.
Since $\mathbf{n}\in M_{[1,s]}$, it follows that
$
\mathbf{u}=\widetilde{\mathbf{u}}+u_1\mathbf{n}
$
lies in the subgroup generated by $M_{[1,s]}$.
This proves the claim.

Now fix an isomorphism
\[
\iota:H\stackrel{\sim}{\longrightarrow}\widehat{H}=\Hom(H,\mathbb{C}^\times),
\]
and set
$
\chi_i=\iota(a_i)
$
for
$1\le i\le k.$
Then for every $\mathbf{u}=(u_1,\dots,u_k)\in \mathbb{Z}^k$ we have
\[
\mathbf{u}\in \mathcal{G}_H(a_1,\dots,a_k)
\quad\Longleftrightarrow\quad
\sum_{i=1}^k u_i a_i=0\text{ in }H
\quad\Longleftrightarrow\quad
\prod_{i=1}^k \chi_i^{u_i}=1.
\]
So under this identification,
\[
\mathcal{G}_H(a_1,\dots,a_k)=G(\chi_1,\dots,\chi_k)
\]
in the notation of \cite{Domokos17}.
Consider the diagonal $H$-module $V$ with weight sequence
$
\chi_1,\dots,\chi_k.
$
By the claim above, for every subset $J\subseteq [1,k]$ with $|J|\le \kappa(H)$,
the group
$
G\bigl((\chi_j)_{j\in J}\bigr)
$
is generated by $M_J$.
Therefore \cite[Theorem~2.1]{Domokos17} implies that the monomials
$
\{\,\mathbf{x}^\mathbf{u}\,|\,\mathbf{u}\in M\,\}
$
form a separating set in $\mathcal{O}(V)^H$.
Applying \cite[Proposition~2.4]{Domokos17}, we conclude that the full relation group
$
\mathcal{G}_H(a_1,\dots,a_k)=G(\chi_1,\dots,\chi_k)
$
is generated by $M$.
\end{proof}
\medskip
As a  consequence, we obtain the following
corollary.

\medskip
\begin{corollary}
\label{cor:d-divisible-lifting}
Let
$
G\cong C_{n_1}\oplus \cdots \oplus C_{n_r}
$
with $
1<n_1\mid \cdots \mid n_r,
$
let
$
G_0=\{g_1,\dots,g_k\}\subseteq G,
$
and let $A$ be a separating atom over $G_0$ with
$
|A|=\beta_{\mathrm{sep}}(G).
$
Let $d\in \mathbb{N}$, and assume
$
d\,\beta_{\mathrm{sep}}(dG)\le |A|-1.
$
If
$
S=\prod_{i=1}^k g_i^{d s_i}\in \mathcal{B}(G_0)
$
for some $s_1,\dots,s_k\in \mathbb{N}_0$, then
$
S\in q\!\bigl(\mathcal{B}(G_0)_{|A|-1}\bigr).
$
\end{corollary}

\begin{proof}
Set
$
a_i=d g_i\in dG
$
for
$
1\le i\le k,
$
and
$
\mathbf{s}=(s_1,\dots,s_k)\in \mathbb{N}_0^k.
$
Since $S$ is zero-sum in $G$,
\[
\sum_{i=1}^k s_i a_i
=
\sum_{i=1}^k d s_i g_i
=
\sigma(S)
=
0,
\]
hence
$
\mathbf{s}\in \mathcal{B}_{dG}(a_1,\dots,a_k)\subseteq \mathcal{G}_{dG}(a_1,\dots,a_k).
$
By Lemma~\ref{lem:uniform-short-generation}, the group
$
\mathcal{G}_{dG}(a_1,\dots,a_k)
$
is generated by the elements
$
\mathbf{u}\in \mathcal{B}_{dG}(a_1,\dots,a_k)
$
with
$
|\mathbf{u}|\le \beta_{\mathrm{sep}}(dG).
$
Therefore there exist
\[
\mathbf{u}^{(1)},\ldots,\mathbf{u}^{(p)},
\ \mathbf{v}^{(1)},\dots,\mathbf{v}^{(q)}
\in \mathcal{B}_{dG}(a_1,\dots,a_k)
\]
such that
$
|\mathbf{u}^{(\alpha)}|,\ |\mathbf{v}^{(\beta)}|\le \beta_{\mathrm{sep}}(dG)
$
for all $1\leq \alpha\leq p, 1\leq \beta\leq q$, and
\begin{equation}
\label{DL1}
\mathbf{s}+\mathbf{u}^{(1)}+\cdots+\mathbf{u}^{(p)}
=\mathbf{v}^{(1)}+\cdots+\mathbf{v}^{(q)}
\quad\text{in }\mathbb{Z}^k.
\end{equation}
Indeed, since $\mathbf{s}$ lies in the
subgroup of $\mathbb{Z}^k$ generated by these short nonnegative relations,
we can write $\mathbf{s}=\sum_j \alpha_j \mathbf{w}^{(j)}$ with $\alpha_j\in\mathbb{Z}$,
$\mathbf{w}^{(j)}\in\mathcal{B}_{dG}(a_1,\dots,a_k)$ and
$|\mathbf{w}^{(j)}|\le \beta_{\mathrm{sep}}(dG)$; splitting the coefficients into positive and negative parts yields
\eqref{DL1}.
For each $\alpha$ and $\beta$, define sequences over $G_0$ by
\[
U^{(\alpha)}
=
\prod_{i=1}^k g_i^{d\,u_i^{(\alpha)}},
\qquad
V^{(\beta)}
=
\prod_{i=1}^k g_i^{d\,v_i^{(\beta)}}.
\]
Since
$
\sum_{i=1}^k u_i^{(\alpha)} a_i=0
$
and
$
\sum_{i=1}^k v_i^{(\beta)} a_i=0
$
in $dG$, both $U^{(\alpha)}$ and $V^{(\beta)}$ belong to $\mathcal{B}(G_0)$.
Moreover,
\[
|U^{(\alpha)}|
=
d\,|\mathbf{u}^{(\alpha)}|
\le
d\,\beta_{\mathrm{sep}}(dG)
\le
|A|-1,
\]
and similarly
$
|V^{(\beta)}|\le |A|-1.
$
Thus
$
U^{(\alpha)},\ V^{(\beta)}\in \mathcal{B}(G_0)_{|A|-1}.
$
Finally, \eqref{DL1} implies the identity of sequences
$
S\cdot U^{(1)}\cdots U^{(p)}=V^{(1)}\cdots V^{(q)}.
$
Hence, in the quotient group $q(\mathcal{B}(G_0))$,
\[
S
=
V^{(1)}\cdots V^{(q)}
\cdot
\bigl(U^{(1)}\cdots U^{(p)}\bigr)^{-1}
\in
q\!\bigl(\mathcal{B}(G_0)_{|A|-1}\bigr).
\]
This proves the corollary.
\end{proof}

\medskip
Specializing Corollary \ref{cor:d-divisible-lifting} to $d=n_1$,
we immediately obtain the following crucial membership
for the $n_1$-th power of our separating atom.

\medskip
\begin{corollary}
\label{cor:n1-power-lifting}
Let
$
G\cong C_{n_1}\oplus \cdots \oplus C_{n_r}$
with
$
1<n_1\mid \cdots \mid n_r,
$
let
$
G_0=\{g_1,\dots,g_k\}\subseteq G,
$
and let $A=\prod_{i=1}^k g_i^{m_i}$ be a separating atom over $G_0$ with
$
|A|=\beta_{\mathrm{sep}}(G).
$
If
$
n_1\,\beta_{\mathrm{sep}}(n_1G)\le |A|-1,
$
then
$
A^{n_1}\in q\!\bigl(\mathcal{B}(G_0)_{|A|-1}\bigr).
$
\end{corollary}

\begin{proof}
Apply Corollary~\ref{cor:d-divisible-lifting} with
$
d=n_1
$
and
$
S=A^{n_1}.
$
Indeed,
$
A^{n_1}=\prod_{i=1}^k g_i^{n_1 m_i}\in \mathcal{B}(G_0),
$
and every multiplicity in $A^{n_1}$ is divisible by $n_1$.
\end{proof}

\medskip
We now assemble the tools developed in this section
to establish the exact separating Noether number for the even-rank case.

\medskip
\begin{proposition}
\label{pro:even}
Let
$
G\cong C_{n_1}\oplus \cdots \oplus C_{n_r}
$
with
$
1<n_1\mid \cdots \mid n_r$ and
$
r=2s.$
Let $p_1$ be the minimal prime divisor of $n_1$.
Then
$
\beta_{\mathrm{sep}}(G)
=
\frac{n_s}{p_1}+n_{s+1}+\cdots+n_r.
$
\end{proposition}

\begin{proof}
Set
$
n=n_s
$
and
$
T=\sum_{j=s+1}^r n_j.
$
By Lemma \ref{lower-bound},
\begin{equation}
\label{EV1}
\beta_{\mathrm{sep}}(G)\ge T+\frac{n}{p_1}.
\end{equation}
On the other hand, Proposition~\ref{thm:even-half-bound} says,
\begin{equation}
\label{EV2}
\beta_{\mathrm{sep}}(G)\le T+\frac{n}{2}.
\end{equation}

If $p_1=2$, then \eqref{EV1} and \eqref{EV2} coincide, so we are done.
Assume from now on that
$
p_1>2.
$
Suppose for contradiction that
$
\beta_{\mathrm{sep}}(G)>T+\frac{n}{p_1}.
$
Choose, by Lemma \ref{reduced}, a separating atom
$
A
$
over some subset $G_0\subseteq G$ such that
$
|G_0|\le r+1
$
and
$
|A|=\beta_{\mathrm{sep}}(G).
$
Then
\begin{equation}
\label{EV3}
T+\frac{n}{p_1}<|A|\le T+\frac{n}{2}.
\end{equation}
By Proposition~\ref{thm:A6-free},
there exists an integer $m\in \mathbb{N}$ such that
\begin{equation}
\label{EV4}
\gcd(m,n_1)=1
\qquad\text{and}\qquad
A^m\in q\!\bigl(\mathcal{B}(G_0)_{|A|-1}\bigr).
\end{equation}
If $n_1G\neq \{0\}$, then Lemma~\ref{lem:n1-rank-shift} gives
$
n_1\,\beta_{\mathrm{sep}}(n_1G)\le T.
$
If $n_1G=\{0\}$, then by convention $\beta_{\mathrm{sep}}(n_1G)=1$, and since necessarily
$
n_1=n_2=\cdots=n_r=n,
$
we have
$
T=sn\ge n_1=n_1\,\beta_{\mathrm{sep}}(n_1G).
$
So in all cases,
$
n_1\,\beta_{\mathrm{sep}}(n_1G)\le T.
$
Since \(T<|A|\) by \eqref{EV3} and both quantities are integers, we have
$
T\le |A|-1.
$
Therefore
$
n_1\,\beta_{\mathrm{sep}}(n_1G)\le |A|-1.
$
Hence Corollary~\ref{cor:n1-power-lifting} gives
\begin{equation}
\label{EV5}
A^{n_1}\in q\!\bigl(\mathcal{B}(G_0)_{|A|-1}\bigr).
\end{equation}
Since $\gcd(m,n_1)=1$, there exist integers $\lambda_1,\lambda_2\in \mathbb{Z}$ such that
$
\lambda_1 n_1+\lambda_2 m=1.
$
Equations \eqref{EV4} and \eqref{EV5} imply
$
A
=
(A^{n_1})^{\lambda_1}(A^m)^{\lambda_2}
\in
q\!\bigl(\mathcal{B}(G_0)_{|A|-1}\bigr),
$
contradicting the fact that $A$ is a separating atom.
Therefore our assumption was false, and
$
\beta_{\mathrm{sep}}(G)\le T+\frac{n}{p_1}.
$
Combining this with \eqref{EV1}, we obtain
$
\beta_{\mathrm{sep}}(G)
=
T+\frac{n}{p_1}
=
\frac{n_s}{p_1}+n_{s+1}+\cdots+n_r.
$
\end{proof}

Combining Propositions \ref{thm:odd-unconditional} and \ref{pro:even}, the proof of Theorem~\ref{formulas} is now complete.
As an application of our explicit formulas and the structural properties of maximal separating atoms, we now prove the conclusion conjectured in \cite[Conjecture~5.4]{SZZ2509} for all groups of rank at least $2$. We now prove Corollary~\ref{cor:conj54-resolved}.

\begin{proof}[Proof of Corollary~\ref{cor:conj54-resolved}]
Set
$
G_1=\supp(A)
$
and
$
k=|G_1|.
$
By definition, the terms of $A$ consist exactly of the
elements in $G_1$, so we may write
$
G_1=\{g_1,\dots,g_k\}
$
and
$
A=\prod_{i=1}^k g_i^{m_i}
$
with
$
m_i\ge 1
$
for
$1\le i\le k.
$
We see that $A$ is also a separating atom over $G_1$.
Now we treat odd and even rank separately.

\smallskip
\noindent
\textbf{Case 1}: $r$ is odd.
Let
\[
r=2s-1,
\qquad
n=n_s,
\qquad
T=n_{s+1}+\cdots+n_r.
\]
For each $i\in [1,k]$, write
\[
m_i=nk_i+x_i,
\qquad
k_i\in \mathbb{N}_0,
\quad
0\le x_i\le n-1.
\]
Set
\[
h_i=ng_i\in nG,
\qquad
B=\prod_{i=1}^k h_i^{k_i}.
\]
Since $A$ is separating over $G_1$ and $|A|=\beta_{\mathrm{sep}}(G)$, Lemma~\ref{lem:B-geodesic} gives
$
|B|\le \mathsf{D}^*(nG)-1.
$
Because
$
n\bigl(\mathsf{D}^*(nG)-1\bigr)=T-(s-1)n,
$
we obtain
\begin{equation}
\label{1}
|A|
=
n|B|+\sum_{i=1}^k x_i
\le
T-(s-1)n+\sum_{i=1}^k x_i.
\end{equation}
Next choose $(t_1,\dots,t_k)\in \mathbb{N}^k$
with $\sum_{i=1}^kt_i$ minimal subject to
$
\sum_{i=1}^k (t_i n-x_i)g_i=0,
$
and define
\[
V=\prod_{i=1}^k g_i^{t_i n-x_i},
\qquad
Y=\prod_{i=1}^k h_i^{t_i-1}.
\]
Then
$
AV=\prod_{i=1}^k g_i^{(k_i+t_i)n}
$
is a zero-sum sequence all of whose multiplicities are divisible by $n$.
By Lemma~\ref{lem:n-divisible-lifting},
$
AV\in q\!\bigl(\mathcal{B}(G_1)_{|A|-1}\bigr).
$
If $|V|\le |A|-1$, then
$
V\in \mathcal{B}(G_1)_{|A|-1}\subseteq q\!\bigl(\mathcal{B}(G_1)_{|A|-1}\bigr),
$
so
$
A=(AV)\cdot V^{-1}\in q\!\bigl(\mathcal{B}(G_1)_{|A|-1}\bigr),
$
contradicting that $A$ is separating over $G_1$.
Hence
$
|V|\ge |A|.
$
By Lemma~\ref{lem:Y-geodesic} (taking $l=1$),
$
|Y|\le \mathsf{D}^*(nG)-1.
$
Therefore
\begin{equation}
\label{3}
|V|
=
n|Y|+kn-\sum_{i=1}^k x_i
\le
T-(s-1)n+kn-\sum_{i=1}^k x_i.
\end{equation}
Combining \eqref{1}, \eqref{3}, and the fact $|V|\geq |A|$,
we get
$
2|A|
\le
2T+(k-2s+2)n.
$
By Proposition~\ref{thm:odd-unconditional},
$
|A|=\beta_{\mathrm{sep}}(G)=T+n,
$
yielding
$
k\ge 2s=r+1.
$
On the other hand,
$
k=|\supp(A)|\le |G_0|\le r+1.
$
Therefore
$
|\supp(A)|=|G_0|=r+1.
$

\smallskip
\noindent
\textbf{Case 2}: $r$ is even.
Let
\[
r=2s,
\qquad
n=n_s,
\qquad
T=n_{s+1}+\cdots+n_r.
\]
For each $i\in [1,k]$, write
\[
m_i=nk_i+x_i,
\qquad
k_i\in \mathbb{N}_0,
\quad
0\le x_i\le n-1,
\]
and set
\[
h_i=ng_i\in nG,
\qquad
B=\prod_{i=1}^k h_i^{k_i}.
\]
Again Lemma~\ref{lem:B-geodesic} yields
$
|B|\le \mathsf{D}^*(nG)-1.
$
Since
$
n\bigl(\mathsf{D}^*(nG)-1\bigr)=T-sn,
$
we get
\begin{equation}
\label{5}
|A|
=
n|B|+\sum_{i=1}^k x_i
\le
T-sn+\sum_{i=1}^k x_i.
\end{equation}
Next choose $(t_1,\dots,t_k)\in \mathbb{N}^k$ with
$\sum_{i=1}^kt_i$ minimal subject to
$
\sum_{i=1}^k (t_i n-x_i)g_i=0,
$
and define
\[
V=\prod_{i=1}^k g_i^{t_i n-x_i},
\qquad
Y=\prod_{i=1}^k h_i^{t_i-1}.
\]
As in the odd case,
$
AV=\prod_{i=1}^k g_i^{(k_i+t_i)n}
$
is $n$-divisible and zero-sum, so Lemma~\ref{lem:n-divisible-lifting} gives
$
AV\in q\!\bigl(\mathcal{B}(G_1)_{|A|-1}\bigr).
$
If $|V|\le |A|-1$, then
$
A=(AV)\cdot V^{-1}\in q\!\bigl(\mathcal{B}(G_1)_{|A|-1}\bigr),
$
a contradiction.
Thus
$
|V|\ge |A|.
$
By Lemma~\ref{lem:Y-geodesic} (taking $l=1$),
$
|Y|\le \mathsf{D}^*(nG)-1,
$
whence
\begin{equation}
\label{7}
|V|
=
n|Y|+kn-\sum_{i=1}^k x_i
\le
T-sn+kn-\sum_{i=1}^k x_i.
\end{equation}
Combining \eqref{5}, \eqref{7}, and
the fact $|V|\geq|A|$, we obtain
$
2|A|
\le
2T+(k-2s)n.
$
Let $p_1$ be the minimal prime divisor of $n_1$.
By Proposition~\ref{pro:even},
$
|A|=\beta_{\mathrm{sep}}(G)=T+\frac{n}{p_1}.
$
We get
$
2\Bigl(T+\frac{n}{p_1}\Bigr)\le 2T+(k-2s)n,
$
which simplifies to
$
\frac{2}{p_1}\le k-2s.
$
Since $k-2s$ is an integer and $\frac{2}{p_1}>0$, it follows that
$
k-2s\ge 1,
$
hence
$
k\ge 2s+1=r+1.
$
Since
$
k=|\supp(A)|\le |G_0|\le r+1,
$
we have
$
|\supp(A)|=|G_0|=r+1.
$
This completes the proof.
\end{proof}


\begin{thebibliography}{99}
\bibitem{Cz19}
K. Cziszter, The Noether number of $p$-groups, J. Algebra Appl.
18 (2019) 1950066.

\bibitem{Cz14}
K. Cziszter and M. Domokos, The Noether number for the groups with a cyclic subgroup of index two, J. Algebra 399
(2014) 546--560.


\bibitem{Cz18}
K. Cziszter, M. Domokos,
and I. Sz\"{o}ll\H{o}si, The Noether numbers and the Davenport constants of the groups of order less
than 32, J. Algebra 510 (2018) 513--541.

\bibitem{DerksenKemper02}
H. Derksen and G. Kemper,
Computational Invariant Theory,
Encyclopaedia of Mathematical Sciences 130,
Springer, Berlin, 2002.

\bibitem{Domokos17}
M. Domokos,
Degree bound for separating invariants of Abelian groups,
Proc. Amer. Math. Soc. 145 (2017), 3695--3708.

\bibitem{GaoGeroldinger06}
W. Gao and A. Geroldinger,
Zero-sum problems in finite abelian groups: a survey,
Expo. Math. 24 (2006), 337--369.

\bibitem{GeroldingerHK06}
A. Geroldinger and F. Halter-Koch,
Non-Unique Factorizations. Algebraic, Combinatorial and Analytic Theory,
Pure and Applied Mathematics 278,
Chapman \& Hall/CRC,  2006.



\bibitem{KL}
B. Klopsch and V. F. Lev,
Generating abelian groups by addition only,
Forum Math. 21 (2009),   23--41.

\bibitem{KohlsKraft10}
M. Kohls and H. Kraft,
Degree bounds for separating invariants,
Math. Res. Lett. 17 (2010),   1171--1182.


\bibitem{Noether16}
E. Noether,
Der Endlichkeitssatz der Invarianten endlicher Gruppen,
Math. Ann. 77 (1916), 89--92.


\bibitem{OlsonI}
J. E. Olson,
A combinatorial problem on finite Abelian groups. I,
J. Number Theory  1 (1969), 8--10.

\bibitem{OlsonII}
J. E. Olson,
A combinatorial problem on finite Abelian groups. II,
J. Number Theory  1 (1969), 195--199.


\bibitem{SchJCTA}
B. Schefler,
The separating Noether number of Abelian groups of rank two,
J. Combin. Theory Ser. A 209 (2025)
105951.

\bibitem{Sch25}
B. Schefler,
The separating Noether number of the direct sum of several copies of a cyclic group,
Proc. Amer. Math. Soc. 153 (2025), 69--79.

\bibitem{SZZ}
B. Schefler, K. Zhao,  and Q. Zhong,
On the separating Noether number of finite abelian groups,
European J. Combin. 133 (2026), 104302.



\bibitem{SZZ2509}
B. Schefler, K. Zhao, and Q. Zhong,
On separating sets of polynomial invariants
of finite abelian group actions,
arXiv:2509.16097, 2025.


\bibitem{Schmid91}
B. Schmid,
Finite groups and invariant theory,
in: Topics in Invariant Theory,
Lecture Notes in Math.
Springer,  1478 (1991), 35--66.

\end{thebibliography}
\end{document}